\documentclass{amsart}
\usepackage[margin=1.4in]{geometry}

\usepackage{amsfonts,amsmath,amssymb,amsthm,amscd,amsxtra}
\usepackage{mathtools,mathptmx}
\usepackage{enumerate,epsfig,verbatim}
\usepackage{centernot}
\usepackage{todonotes}
\usepackage[T1]{fontenc}
\usepackage[usenames,dvipsnames]{pstricks}
\usepackage[mathscr]{eucal}
\usepackage[notcite,notref]{}
\usepackage{tikz}
\usetikzlibrary{graphs,graphs.standard,matrix, calc, arrows,quotes}

\usepackage[all,2cell,ps]{xy}
\usepackage[pagebackref]{hyperref}
\usepackage{cleveref}

\usepackage{nicefrac}
\usepackage{array}
\usepackage{xcolor}
\usepackage{color}
\usepackage[all]{xypic}

\usepackage[notcite,notref]{}
\usepackage{url}
\usepackage{datetime}

\usepackage{amssymb,amstext,amsmath,amscd,amsthm,amsfonts,enumerate,graphicx,latexsym,color,stmaryrd}

\usepackage[T1]{fontenc}
\usepackage[usenames,dvipsnames]{pstricks}
\usepackage[mathscr]{eucal}
\usepackage[notcite,notref]{} 

\usepackage[usenames,dvipsnames]{pstricks}
\usepackage[mathscr]{eucal}
\usepackage{amsfonts,amsmath,amssymb,amsthm,amscd,amsxtra}
\usepackage{enumerate,verbatim}
\usepackage[all,2cell,ps]{xy}
\usepackage{mathptmx}
\usepackage[notcite, notref]{}
\usepackage[pagebackref]{hyperref}
\usepackage{mathtools}
\usepackage{nicefrac}
\usepackage{array}
\usepackage{xcolor}
\usepackage[all,2cell,ps]{xy}
\usepackage{amsfonts}
\usepackage[margin=1.4in]{geometry}
\usepackage{color}
\usepackage[notcite,notref]{}
\usepackage{url}
\usepackage{datetime}
\usepackage{mathtools}
\usepackage[all,2cell,ps]{xy}
\usepackage{amssymb}
\usepackage{amsmath}
\usepackage{amsfonts}
\usepackage{amsmath}
\usepackage{amsthm}
\usepackage{amssymb}
\usepackage{amscd}
\usepackage{amsfonts}
\usepackage{amsxtra}     
\usepackage{epsfig}
\usepackage{verbatim}

\SelectTips{cm}{}
\usepackage[all]{xy}
\usepackage{comment,cancel}
\usepackage{tikz}
\usepackage{appendix}
\usepackage{xr}
\usepackage{datetime}

\usepackage{mathptmx}

\newcommand{\comm}[1]{{\color[rgb]{0.0, 0.5, 0.0} #1}}

\newcommand{\new}[1]{{\color{blue} #1}}
\newtheorem{thm}{Theorem}[section]
\newtheorem{prop}[thm]{Proposition}
\newtheorem{cor}[thm]{Corollary}
\newtheorem{lem}[thm]{Lemma}
\theoremstyle{definition}
\newtheorem{chunk}[thm]{\hspace*{-1.065ex}\bf}

\newtheorem{eg}[thm]{Example}

\newtheorem{ques}[thm]{Question}
\newtheorem{rmk}[thm]{Remark}
\theoremstyle{remark}

\newtheorem*{claim*}{Claim}
\numberwithin{equation}{thm}

\newcommand{\NN}{\mathbb{N}}
\newcommand{\ZZ}{\mathbb{Z}}
\newcommand{\CC}{\mathbb{C}}

\newcommand{\iso}{\cong}

\newcommand{\lra}{\longrightarrow}
\newcommand{\ps}[1]{[\![#1]\!]}

\def\fm{\mathfrak{m}}
\def\fn{\mathfrak{n}}
\def\fp{\mathfrak{p}}
\def\fq{\mathfrak{q}}

\def\rH{\mathrm{H}}

\def\p{\mathrm{p}}
\def\q{\mathrm{q}}

\def\cx{\operatorname{\mathrm{cx}}}

\def\depth{\operatorname{\mathrm{depth}}}

\def\Tor{\operatorname{\mathrm{Tor}}}
\def\Spec{\operatorname{\mathrm{Spec}}}
\def\Ann{\operatorname{\mathrm{Ann}}}
\def\Ass{\operatorname{\mathrm{Ass}}}
\def\Supp{\operatorname{\mathrm{Supp}}}

\def\pd{\operatorname{\mathrm{pd}}}

\def\gd{\operatorname{\mathrm{G-dim}}}
\def\cid{\operatorname{\mathrm{CI-dim}}}

\def\CI{\operatorname{\mathrm{CI-dim}}}
\def\Gdim{\operatorname{\mathrm{G-dim}}}

\def\rcid{\operatorname{\mathrm{red-CI-dim}}}
\def\rpd{\operatorname{\mathrm{red-pd}}}
\def\rgd{\operatorname{\mathrm{red-G-dim}}}

\DeclareMathOperator{\Hdim}{\operatorname{\mathbf{\mathbb{I}}}}
\DeclareMathOperator{\rHdim}{\operatorname{\mathsf{red-\mathbb{I}}}}

\def\reddeg{\mathrm{reddeg^{\ast}}}

\def\Ext{\mathrm{Ext}}
\def\Hom{\mathrm{Hom}}
\def\rHom{\mathrm{R}\mathrm{Hom}}
\def\ltensor{\otimes^\mathbb{L}}

\newcommand{\rpm}{\raisebox{.2ex}{$\scriptstyle\pm$}}

\begin{document}

\title[Depth formula for modules of finite reducing projective dimension]{Depth formula for modules of finite \\reducing projective dimension}

\author{Olgur Celikbas}
\address{Olgur Celikbas\\ School of Mathematical and Data Sciences, West Virginia University, 
Morgantown, WV 26506 U.S.A}
\email{olgur.celikbas@math.wvu.edu}

\author{Toshinori Kobayashi}
\address{Toshinori Kobayashi \\ School of Science and Technology, Meiji University, 1-1-1 Higashi-Mita, Tama-ku, Kawasaki-shi, Kanagawa 214-8571, Japan}
\email{tkobayashi@meiji.ac.jp}

\author{Brian Laverty}
\address{Brian Laverty\\School of Mathematical and Data Sciences, West Virginia University,
	Morgantown, WV 26506 U.S.A}
\email{bml0016@mix.wvu.edu}

\author[Hiroki Matsui]{Hiroki Matsui}
\address{Hiroki Matsui\\ Department of Mathematical Sciences,
Faculty of Science and Technology,
Tokushima University,
2-1 Minamijosanjima-cho, Tokushima 770-8506, Japan}
\email{hmatsui@tokushima-u.ac.jp}

\thanks{Hiroki Matsui was partly supported by JSPS Grant-in-Aid for Early-Career Scientists 22K13894}
\thanks{Toshinori Kobayashi was partly supported by JSPS Grant-in-Aid for JSPS Fellows 21J00567}

\maketitle{}

\begin{abstract} We prove that the depth formula holds for two finitely generated Tor-independent modules over Cohen-Macaulay local rings if one of the modules considered has finite reducing projective dimension (for example, if it has finite projective dimension, or the ring is a complete intersection). This generalizes a result of Bergh-Jorgensen which shows that the depth formula holds for two finitely generated Tor-independent modules over Cohen-Macaulay local rings if one of the modules considered has reducible complexity and certain additional conditions hold.

Each module that has reducible complexity also has finite complexity and finite reducing projective dimension, but not necessarily vice versa. So a new advantage we have is that, unlike modules of reducible complexity, Betti numbers of modules of finite reducing projective dimension can grow exponentially. 
\end{abstract}

\section{Introduction}
Throughout $R$ denotes a commutative Noetherian local ring with unique maximal ideal $\fm$ and residue field $k$, and all $R$-modules are assumed to be finitely generated.

In his beautiful paper \cite{Aus} Auslander proved that, if $M$ and $N$ are $R$-modules such that $\pd_R(M)<\infty$, and $M$ and $N$ are Tor-independent, that is, if $\Tor_i^R(M,N)=0$ for all $i\geq 1$, then the following  holds:
\[
\depth_R(M) + \depth_R(N) = \depth(R)+ \depth_R(M\otimes_RN).
\]
Huneke-Wiegand dubbed this depth equality the \emph{depth formula}, and proved that the depth formula holds if $R$ is a complete intersection ring, and $M$ and $N$ are Tor-independent $R$-modules, regardless of whether $M$ or $N$ has finite projective dimension; see \cite[2.5]{HW1} and \cite[page 163]{HW2}. Subsequently, consequences of the depth formula, and different conditions that force this formula to hold were studied in several papers; see, for example, \cite{CeD2, Bounds, CJD}. In this paper we are concerned with results of Bergh-Jorgensen \cite{BJD} who studied the depth formula over Cohen-Macaulay rings for modules that have reducible complexity; see also Bergh \cite{Be} and Sadeghi \cite{Sa} for similar results over local rings that are not necessarily Cohen-Macaulay. The results of Bergh-Jorgensen in \cite{BJD} concerning the depth formula can be summarized as follows; see also \ref{thmBJ2} and Remark \ref{dfinfo}.

\begin{thm} (Bergh-Jorgensen; see \cite[3.3, 3.4, 3.5, 3.6]{BJD}) \label{thmintroBJ} Let $R$ be a Cohen-Macaulay local ring and let $M$ and $N$ be Tor-independent $R$-modules. Assume at least one of the following holds:
\begin{enumerate}[\rm(i)] 
\item $\depth_R(M) < \depth(R)$.
\item $N$ has finite Gorenstein dimension (for example, $R$ is Gorenstein).
\item $N$ has reducible complexity (for example, $N$ is periodic, or $R$ is a complete intersection).
\end{enumerate}
If $M$ has reducible complexity, then the depth formula holds for $M$ and $N$. \qed
\end{thm}

In this paper we generalize Theorem \ref{thmintroBJ} and prove the following:

\begin{thm} \label{thmintro} Let $R$ be a Cohen-Macaulay local ring and let $M$ and $N$ be Tor-independent $R$-modules. If  $M$ has finite reducing projective dimension (for example, $M$ has reducible complexity), then the depth formula holds for $M$ and $N$. \qed 
\end{thm}



%

Our main result, namely Theorem \ref{mainthm}, is more general than Theorem \ref{thmintro}; see also Proposition \ref{mainprop}. Theorem \ref{mainthm} establishes the derived version of the depth formula \cite{Foxby} over Cohen-Macaulay local rings provided that Tor modules eventually vanish and one of the modules considered has finite reducing projective dimension. The gist of the proof of Theorem \ref{mainthm} relies upon the usage of some derived category tools; see, for example, \ref{p27} and \ref{rmkmain}. 

We recall the definitions of reducible complexity and reducing projective dimension in the next section; see \ref{rcx} and \ref{rdim}. 
The definition of reducing invariants \cite{CA} was motivated by that of reducible complexity \cite{Be}, but these two definitions are different in nature: The latter deals with modules that have finite complexity, namely, modules that have polynomial growth on their Betti numbers.
Each module of finite reducing projective dimension has reducible complexity by definition, but the converse of this fact is not true in general. 
For example, if $R$ is a Cohen-Macaulay local ring with minimal multiplicity and $|k|=\infty$ (for example, $R=\CC[\![t^3, t^4, t^5]\!]$) and $M=\Omega^n_R(k)$ for some $n \geq 0$, then $M$ has finite reducing projective dimension, but $M$ does not have reducible complexity unless $R$ is a hypersurface; see \cite[1.2]{RD4}. There are many similar examples in the literature; see, for example, Examples \ref{ahlat1}, \ref{ahlat2}, and \ref{ahlat3}.

%

In section 2 we recall the necessary definitions and provide several examples. The proof of Theorem \ref{thmintro} is given in section 3, while section 5 contains the proofs of the auxiliary results we use to prove the theorem. Sections 4 is devoted to some applications of our results; see, for example, Corollary \ref{cor5.3} and Proposition \ref{rmk5.4}. In Section 4, besides other applications, we consider a beautiful formula of Jorgensen \cite{JAB} which extends the classical Auslander-Buchsbaum formula; see \ref{JDF}. Over Cohen-Macaulay rings, we improve Jorgensen's formula and prove the following; see Theorem \ref{thm5.5}.

\begin{thm} \label{thm5.5} \label{propnewintro} Let $R$ be a Cohen-Macaulay local ring and let $M$ and $N$ be nonzero $R$-modules. If $\Tor_i^R(M,N)=0$ for all $i\gg 0$ and $\rcid_R(M)<\infty$, then 
\begin{align*}
\sup\{i:\Tor_i^R(M,N)\neq 0\} & = \sup \Bigl\{  \depth(R_{\fp})-\depth_{R_\fp}(M_\fp) -\depth_{R_\fp}(N_\fp) \mid \fp \in \Spec(R)  \Bigr\}.
\end{align*} 
\end{thm}

%

%


\section{Definitions, preliminary results, and examples}

In this section we record some definitions, and discuss several examples and preliminary results that are needed in the subsequent sections.

\begin{chunk} (\textbf{Gorenstein and complete intersection dimensions \cite{AuBr, AGP}}) \label{G} 

An $R$-module $M$ is said to be \emph{totally reflexive} if the natural map $M \to M^{\ast\ast}$, where $M^{\ast}=\Hom_R(M,R)$, is bijective and $\Ext^i_R(M,R) = 0 = \Ext^i_R(M^{\ast},R)$ for all $i\geq 1$. The infimum of $n$ for which there exists an exact sequence $0\to X_n \to \cdots \to X_0 \to M \to 0$, such that each $X_i$ is totally reflexive, is called the \emph{Gorenstein dimension} of $M$. If $M$ has Gorenstein dimension $n$, we write $\Gdim_R(M) = n$. Therefore, $M$ is totally reflexive if and only if $\Gdim_R(M)\leq 0$, where it follows by convention that $\Gdim_R(0)=-\infty$. 

A diagram of local ring maps $R \to R' \twoheadleftarrow S$  is called a \emph{quasi-deformation} if $R \to R'$ is flat and the kernel of the surjection $R' \twoheadleftarrow S$ is generated by a regular sequence on $S$. The \emph{complete intersection dimension} of $M$ is defined as follows:
\begin{equation}\notag{}
\CI_R(M)=\inf\{ \pd_S( M\otimes_RR') -\pd_S(R'): R \to R' \twoheadleftarrow S \text{ is a quasi-deformation}\}.
\end{equation}

Note that, if $R$ is a complete intersection ring, then each $R$-module has finite complete intersection dimension \cite[1.3]{AGP}. Moreover, the inequalities $\Gdim_R(M)\leq \CI_R(M) \leq \pd_R(M)$ hold in general; if one of these dimensions is finite, then it is equal to the one to its left; see \cite[1.4]{AGP}. 

We set $\pd_R(0)=\Gdim_R(0)=\cid_R(0)=-\infty$ and also $\depth_R(0)=\infty$. \qed
\end{chunk}

The reducing projective dimension definition \cite{CA} originated from the reducible complexity definition, a notion introduced by Bergh; see \cite{Be}. Hence we first recall the definition of reducible complexity. In the following $\Omega_R^i M$ denotes the $i$-th syzygy  of a given $R$-module $M$ in its minimal free resolution.

\begin{chunk} (Complexity and reducible complexity \cite{Av1, Be}) \label{rcx} Let $M$ be an $R$-module. The \emph{complexity} of $M$ \cite{Av1} is:
\[
\cx_R(M) = \inf\{r \in \NN_0~|~\exists\; A \in \mathbb{R} \text{ such that } \dim_k\big(\Tor_n^R(M,k)\big) \leq A \cdot n^{r-1} \text{ for all } n\gg 0\}.
\]

Note that $\cx_R(M)=0$ if and only if $\pd_R(M)<\infty$, and $\cx_R(M)\leq 1$ if and only if $M$ has bounded Betti numbers. In general $\cx_R(M)$ may be infinite; in fact, $R$ is a complete intersection if and only if $\cx_R(k)<\infty$ \cite[8.1.2]{Av2}. Moreover, if $\CI_R(M)<\infty$ (for example, if $R$ is a complete intersection ring), then $\cx_R(M)<\infty$; see \cite[5.3]{AGP}.

The module $M$ is said to have \emph{reducible complexity}  \cite{Be} if $\cx_R(M)=0$, or if $0<\cx_R(M)=r<\infty$ and there are short exact sequences of $R$-modules
\[
\{0\to K_{i-1}\to K_i \to\Omega_R^{n_i} K_{i-1} \to 0\}^{r}_{i=0},
\]
with $n_i\geq 0$, where $K_0=M$ and $\cx_R(K_{i+1})=\cx_R(K_{i})-1$ for all $i=0, \ldots, r$ (so that $\pd_R(K_r)<\infty$).

There are many examples of modules that have reducible complexity. For example, periodic modules over arbitrary local rings, and modules over complete intersection rings (or more generally, modules over arbitrary local rings that have finite complete intersection dimension) have reducible complexity; see \cite{AGP, Be} for the details.
\end{chunk}


\begin{chunk} (Reducing invariants \cite{CA, AT}) \label{rdim}  Let $M$ be an $R$-module and let $\Hdim_R$ denote an invariant of $R$-modules, that is, $\Hdim_R$ denotes a function from the set of isomorphism classes of $R$-modules to the set $\ZZ \cup \{ \rpm \infty \}$. Classical examples of $\Hdim_R$ are the projective dimension $\pd_R$, the Gorenstein dimension $\Gdim_R$, and the complete intersection dimension $\cid_R$.

The \emph{reducing invariant} $\rHdim_R(M)$ of $M$ is zero if $\Hdim_R(M)<\infty$. Hence, if $\Hdim_R \in \{\pd_R, \Gdim_R, \cid_R\}$, then $\rHdim_R(0)=0$ since $\Hdim_R(0)=-\infty$; see \cite{Au2}. If $\Hdim_R(M)=\infty$, we write $\rHdim_R(M)<\infty$ provided that there exist integers $r\geq 1$, $a_i\geq 1$, $b_i\geq 1$, $n_i\geq 0$, and short exact sequences of $R$-modules of the form 
\begin{equation}\tag{\ref{rdim}.1}
0 \to K_{i-1}^{{\oplus a_i}} \to K_{i} \to \Omega_R^{n_i}K_{i-1}^{{\oplus b_i}} \to 0,
\end{equation}
for each $i=1, \ldots, r$, where $K_0=M$ and $\Hdim_R(K_r)<\infty$.  If a sequence as in (\ref{rdim}.1) exists, then we call $\{K_0,  \ldots, K_r\}$ a \emph{reducing $\Hdim$-sequence} of $M$; see \cite[2.1]{CA} and \cite[2.5]{AT}.


If $\Hdim_R(M)=\infty$, we define:\begin{equation}\notag{}
\rHdim_R(M)=\inf\{ r\in \NN: \text{there is a reducing $\Hdim$-sequence }  \{K_0,  \ldots, K_r\} \text{ of }  M\}. 
\end{equation}
It follows that  $0\leq \rHdim_R(M)\leq \infty$. 
\qed
\end{chunk}

\begin{rmk} \label{lemrdim} In the definition of reducible complexity and reducing invariants, minimal syzygy modules are used. However, if $\Hdim_R \in \{\pd_R, \Gdim_R, \cid_R\}$, then we may use syzygies that are not necessarily minimal to define $\rHdim_R$. This can be seen as follows:

If $0 \to K \to K_{1} \to \Omega_R^{n}K \oplus F \to 0$ is a short exact sequence of $R$-modules, where $F$ is a free module, then we obtain an induced exact sequence $0 \to K \to K'_{1} \to \Omega_R^{n}K \to 0$, where $K_1\cong K'_{1}  \oplus F$ and $\rHdim_R(K_1)=\rHdim_R(K_1')$; see \cite[3.8 and the proof of the claim in the proof of 4.2]{RD4}. Hence, if $M$ is an $R$-module, then the value of $\rHdim_R(M)$ does not change whether or not we use minimal syzygies to define $\rHdim_R$. \qed
\end{rmk}

In passing we make use of \Cref{lemrdim} and note:

\begin{lem} \label{lemrpd} Let $M$ be an $R$-module and let $\Hdim_R \in \{ \pd_R, \Gdim_R,\cid_R\}$. Then $\rHdim_{R_\fp}(M_\fp) \le \rHdim_R(M)$ for  all $\fp \in \Spec(R)$. 
\end{lem}

\begin{proof} If $\fp \in \Spec(R)$, then the localization of a reducing $\Hdim$-sequence of $M$ over $R$ at $\fp$ is a reducing $\Hdim$-sequence of $M_{\fp}$ over $R_{\fp}$; see \ref{rdim} and \Cref{lemrdim}. So $\rHdim_{R_\fp}(M_\fp) \le \rHdim_R(M)$ for all $\fp \in \Spec(R)$. 
\end{proof}

Next are some further results that are needed later in the sequel.

\begin{chunk} \label{rdim1} The following properties are due to the definition of reducing invariants; see \ref{rdim} and also \cite[2.2]{RD4}:
\begin{enumerate}[\rm(i)]
\item If $1\leq \rHdim_R(M)<\infty$, then there is an exact sequence $0\to M^{\oplus a} \to K \to \Omega^n_R M^{\oplus b}\to 0$, where $a\geq 1$, $b\geq 1$, and $n\geq 0$ are integers, and $K$ is an $R$-module such that $\rHdim_R(K)=\rHdim_R(M)-1$.
\item Conversely, if $0\to M^{\oplus a} \to K \to \Omega^n_RN^{\oplus b}\to 0$ is an exact sequence of $R$-modules, where $a\geq 1$, $b\geq 1$, and $n\geq 0$ are integers, then $\rHdim_R(M)\leq \rHdim_R(K)+1$. 
\end{enumerate}
\end{chunk}

\begin{chunk}\label{rdim2} Let $M$ be an $R$-module. 
\begin{enumerate}[\rm(i)]
\item $\rcid_R(M)<\infty$ if and only if $\rpd_R(M)<\infty$; see \cite[2.7]{CDKM}.
\item If $x\in \fm$ is a non zero-divisor on $R$ and $M$, then $\rpd_{R/xR}(M/xM)\leq\rpd_R(M)$; see \cite[3.4]{RD4}. 
\end{enumerate}
\end{chunk}

\begin{rmk} \phantom{}
\begin{enumerate}[\rm(i)]
\item The definition of the reducing dimension we use in this paper is taken from \cite{AT}, even though this notion was originally defined in \cite{CA}. The difference between these two definitions is that \cite{AT} only requires the integers $n_i$ in Definition \ref{rdim} to be nonnegative, while \cite{CA} requires these numbers to be positive. Hence, if a module has finite reducing invariant with respect to the definition given in \cite{CA}, then it has also finite reducing invariant with respect to the definition we use in this paper. 
\item Given an integer $c\geq 1$, one can always find a ring $R$ and an $R$-module $M$ with $\rpd_R(M)=c<\infty$ and $\pd_R(M)=\infty$. More precisely, if $R$ is a singular complete intersection ring of codimension $c$ (for example, $R = k\ps{x_1,\ldots,x_c,y_1,\ldots , y_c}/(x_1y_1,\ldots , x_cy_c)$ for some field $k$) and $M=\Omega_R^i(k)$ for an integer $i\geq 0$, then $\pd_R(M)=\infty$ and $\rpd_R(M)=\cx_R(M)=c$; see \cite[2.8]{RD4}. 
\item Several characterizations of local rings in terms of reducing dimensions were obtained in \cite{CDKM}. For example, it follows that a local ring $R$ is Gorenstein if and only if each $R$-module has finite reducing Gorenstein dimension; see \cite[3.12]{CDKM}. Similarly there are various questions that remain open about these invariants. For example, it is not known whether or not the residue field of a local ring always has finite reducing projective dimension; see \cite{RD4}. \qed
\end{enumerate}
\end{rmk}


In general, a module having reducible complexity also has finite reducing projective dimension; see \ref{rcx} and \ref{rdim}. On the other hand, modules that have finite reducing projective dimension -- which have infinite complexity -- are abundant in the literature. In other words, reducing projective dimension is a finer invariant than reducible complexity. Next we record several examples which illustrate this fact. 

\begin{eg} (\cite[2.3 and 2.4]{CA}) \label{ahlat1} Let $R=k\ps{x,y}/(x^2,xy,y^2)$. The beginning of the minimal free resolution of $k$ is:
\[
\cdots\lra R^{\oplus 4} \xrightarrow{\begin{pmatrix}x\;y\;0\;0\\0\;0\;x\;y\end{pmatrix}} R^{\oplus 2} \xrightarrow{\begin{pmatrix}x\; y\end{pmatrix}} R \lra k \lra 0.
\]
Notice $R$ is not Gorenstein so that we have $\cx_R(k)=\CI_R(k)=\Gdim_R(k)=\pd_R(k)=\infty$; see \ref{G} and \ref{rcx}. As $\fm \cong k^{\oplus 2}$, it follows that $\Omega_R^2k\iso k^{\oplus 4}$. So, there is an exact sequence $0 \to k^{\oplus 4} \to R^{\oplus 2} \to \Omega_R k \to 0$. Therefore $\{k, R^{\oplus 2}\}$ is a reducing $\pd$-sequence of $k$. We see that $\rpd_R(k)=\rgd_R(k)=1<\infty$. 

This construction can be generalized as long as $\fm^2=0$.  In that case there is an exact sequence $0 \to k^{\oplus e^2} \to R^{\oplus e}\to \Omega_R k\to 0$, where $e$ is the embedding dimension of $R$. This fact implies that $\{k, R^{\oplus e}\}$ is a reducing pd-sequence (and hence reducing $\gd$) sequence of $k$. If $R$ is not Gorenstein, then $\rpd_R(k)=1=\rgd_R(k)$ and $\pd_R(k)=\infty=\gd_R(k)$.
\end{eg}

\begin{eg} (\cite[2.7]{CA}) \label{ahlat2} Let $R=k\ps{x^3,x^2y,xy^2,y^3}$, the 3rd Veronese subring of the formal power series ring $k\ps{x,y}$, and consider the $R$-module $M=(x^2,xy,y^2)$. It follows that $\gd_R(M)=\infty$, and there exists a short exact sequence $0\to M^{\oplus 4}\to F\to \Omega_R M\to 0$ for some free $R$-module $F$. Therefore we have that $\rgd_R(M)= 1$.
\end{eg}

\begin{eg} \label{ahlat3}  (\cite[2.12]{RD4}) Let $R$ be a Cohen-Macaulay local ring which is not regular. If $I$ is an Ulrich ideal of $R$ which is not a parameter ideal, then $\pd_R(R/I)=\infty$ and $\rpd_R(R/I)=1$. 

If $R$ has minimal multiplicity and $|k|=\infty$, then $\fm$ is an Ulrich ideal which is not a parameter ideal so that $\pd_R(k)=\infty$ and $\rpd_R(k)=1$. Some specific examples can be given as follows:
\begin{enumerate}[\rm(i)]
\item If $R=k\ps{x,y,z}/(x^3-y^2, z^2-x^2y)$ and $I = (x,y)$. Then $R$ is Cohen-Macaulay but not regular, and $I$ is an Ulrich ideal of $R$ which is not a parameter ideal. So $\pd_R(R/I)=\infty$ and $\rpd_R(R/I)=1$.
\item Let $R=\CC\ps{x,y}/(x,y)^2$, or $R=\CC\ps{t^3,t^4,t^5}$, or $R=\CC\ps{t^4,t^5,t^6,t^7}$. Then $R$ is a non-regular Cohen-Macaulay local ring with minimal multiplicity. It follows that $\pd_R(k)=\infty$ and $\rpd_R(k)=1$. \qed
\end{enumerate}
\end{eg}

\begin{chunk} Given $R$-modules $M$ and $N$, we set $\q^R(M,N)=\sup\{i: \Tor_i^R(M,N)\neq 0\}$. Note that, if $M$ and $N$ are nonzero $R$-modules, then $\q^R(M,N)=0$ if and only if $M$ and $N$ are Tor-independent.
\end{chunk}

We need several facts about complexes for our arguments.

\begin{chunk} \label{cxp} An $R$-complex $X$ is a complex of (finitely generated) $R$-modules which is indexed homologically.
\begin{enumerate}[\rm(i)]
\item We say $X$ is homologically bounded if $\sup(X)-\inf(X)<\infty$, where $\inf(X) = \inf\{n\mid \rH_n(X)\not=0\}$ and $\sup(X) = \sup\{n \mid \rH_n(X)\not=0\}$. Also the $R$-complex $X[1]$ is defined by $X[1]_n=X_{n-1}$ and $\partial^{X[1]}_n=-\partial^X_{n-1}$. 
\item The depth of $X$ can be defined as $\depth_R (X)=-\sup\left( \rHom_R(k,X)\right)$; see \cite[6.1]{I}. If $X\not\simeq 0$ and $X$ is homologically bounded, then $\depth_R(X) \leq \dim(R)-\sup(X)$ \cite[3.17]{FoxbyFlat}. Hence, if $X=M\ltensor_RN$ for some nonzero $R$-modules $M$ and $N$ such that $\q^R(M,N)<\infty$, then $\depth_R(M\ltensor_RN) \leq \dim(R)$.
\item Assume $s=\sup(X)<\infty$. Then $\depth_R (X) \ge -s$. Moreover, it follows that $\depth_R (X) = -s$ if and only if $\depth_R\big(\rH_{-s}(X)\big)=0$; see 
\cite[3.3 on page 160]{FoxbyFlat}.
\item Let $X \to Y \to Z \to X[1]$ be an exact triangle, where $X$, $Y$, and $Z$ are $R$-complexes such that $\sup(X)<\infty$ and $\sup(Z)<\infty$ (so that $\sup(Y)<\infty$). Then the \emph{derived depth lemma} yields the following; see, for example, \cite[1.2.9]{BH}.
\begin{enumerate}[\rm(a)]
\item $\depth_R(X) \ge \min\{\depth_R(Y), \depth_R(Z)+1\}$.
\item $\depth_R(Y) \ge \min\{\depth_R(X), \depth_R(Z)\}$.
\item $\depth_R(Z) \ge \min\{\depth_R(Y), \depth_R(X)-1\}$.
\item If $\depth_R(Z) \ge \depth_R(X)$, then $\depth_R(Y)=\depth_R(X)$.
\end{enumerate}
\item We say the \emph{derived depth formula} \cite{CJD} holds for given $R$-modules $M$ and $N$ provided that:
$$\depth_R(M)+\depth_R(N) = \depth(R)+\depth_R(M\ltensor_RN).$$
\item Let $M$ and $N$ be $R$-modules. If $\q^R(M,N)<\infty$ and $\CI_R(M)<\infty$, then the derived depth formula holds for $M$ and $N$; see \cite[5.4]{CJD}. This recovers a result of Foxby \cite[2.1]{Foxby} who initially proved that the derived depth formula holds for $M$ and $N$ if $\pd_R(M)<\infty$; see also \cite[2.2]{I}. \qed
\end{enumerate}
\end{chunk}


We finish this section by recording the preliminary results which are needed for the proof of Theorem \ref{mainthm}. To not to disturb the flow of the paper, we prove these preliminary results, namely \ref{lem2.1intro}, \ref{p27}, \ref{rmkmain}, and \ref{rmk:ss}, in section 5.

%


\begin{chunk}\label{lem2.1intro} Let $M$ and $N$ be $R$-modules. Assume there is a short exact sequence of $R$-modules
\begin{equation}\notag{}
0 \to M^{\oplus a} \to K \to \Omega_R^n M^{\oplus b} \to 0,
\end{equation}
where $a\geq 1$, $b\geq 1$, and $n \geq 0$ are integers.
\begin{enumerate}[\rm(i)]
\item If $\depth_R(M) \le \depth(R)$, then $\depth_R(M) = \depth_R(K)$.
\item If $\q^R(M,N)<\infty$, then $\q^R(M,N) = \q^R(K,N)$. 
\item Assume $\q^R(M,N)<\infty$. If $n\geq 1$, or $\depth_R \Big(\Tor^R_{\q^R(M,N)}(M,N)\Big) \le 1$, then
\[
\depth_R\Big(\Tor^R_{\q^R(M,N)}(M,N)\Big) = \depth_R\Big(\Tor^{R}_{\q^R(K,N)}(K,N)\Big). 
\]
\end{enumerate}
\end{chunk}

\begin{chunk} \label{p27} Let $M$ and $N$ be $R$-modules. Assume $\q^R(M,N)<\infty$. Assume further there exists $x\in \fm$ such that $x$ is a non zero-divisor on $R$, $M$, and $N$. Then the following conditions are equivalent:
\begin{enumerate}[\rm(i)]
\item $\depth_R(M)+\depth_R(N) = \depth(R)+\depth_R(M\ltensor_RN)$. 
\item $\depth_{R/xR}(M/xM)+\depth_{R/xR}(N/xN) = \depth(R/xR)+\depth_{R/xR}(M/xM\ltensor_{R/xR}N/xN)$. 
\end{enumerate}
\end{chunk}

\begin{chunk} \label{rmkmain} Let $M$ and $N$ be $R$-modules such that $\q^R(M,N)<\infty$. Assume the following conditions hold:
\begin{enumerate}[\rm(i)]
\item $\depth_R(M)\leq\depth(R)$. 
\item There is an exact sequence of $R$-modules $0 \to M^{\oplus a} \to K \to \Omega_R^n M^{\oplus b} \to 0$, where $a\geq 1$, $b\geq1$, and $n\geq0$ are integers, and the derived depth formula holds for $K$ and $N$.
\end{enumerate}
Then $\depth_R(K \ltensor_R N) = \min \{\depth_R(M \ltensor_R N), \depth_R(N)\}$. \qed 
\end{chunk}

The following useful result is essentially due to Iyengar \cite[2.3]{I}; it allows us to avoid using spectral sequences in the proof of Proposition \ref{mainprop}.

\begin{chunk} \label{rmk:ss} Let $X$ be an $R$-complex such that $s=\sup(X)$. Assume $s<\infty$ and $\depth_R\big(\rH_{s}(X)\big)\leq 1$. Then $\depth_R(X)=\depth_R\big(\rH_{s}(X)\big)-s$. 
\end{chunk}

\section{Main result}

In this section we prove Theorem \ref{thmintro} and determine some conditions that imply the derived depth formula holds. We start by preparing a lemma; recall that we set $\q^R(M,N)=\sup\{i: \Tor_i^R(M,N)\neq 0\}$ for given $R$-modules $M$ and $N$.

\begin{lem}\label{mainlem} Let $M$ and $N$ be $R$-modules such that $\q^R(M,N)<\infty$. Assume there is an exact sequence of $R$-modules $0 \to M^{\oplus a} \to K \to \Omega_R^n M^{\oplus b} \to 0$, where $a\geq 1$, $b\geq 1$, and $n\geq 0$ are integers. Assume the derived depth formula holds for $K$ and $N$. Assume further \emph{at least} one of the following holds:
\begin{enumerate}[\rm(i)]
\item $\depth_R(M) < \depth(R)$.
\item $\depth_R(M)\leq \depth(R)$ and $N$ is maximal Cohen-Macaulay.
\item $\depth_R(M)\leq \depth(R)$ and $\depth_R( \Tor^R_{\q^R(M,N)}(M,N)) \le 1$.
\end{enumerate}
Then the derived depth formula holds for $M$ and $N$.
\end{lem}

\begin{proof} As we assume $\depth_R(M)\leq \depth(R)$, it follows from \ref{lem2.1intro}(i) that $\depth_R(K) = \depth_R(M)$. Also, by our hypothesis, the derived depth formula holds for $K$ and $N$. These facts yield:
\begin{equation}\tag{\ref{mainlem}.1}
 \depth_R(N) - \depth (K \ltensor_R N) =  \depth(R)-\depth_R(M). 
\end{equation}
So, in view of (\ref{mainlem}.1), the derived depth formula holds for $M$ and $N$ if $\depth(M \ltensor_R N)=\depth_R(K \ltensor_R N)$. Moreover, by  \ref{rmkmain}, we have:
\begin{equation}\tag{\ref{mainlem}.2}
\depth_R(K \ltensor_R N) = \min \{\depth_R(M \ltensor_R N), \depth_R(N)\}.
\end{equation}

Assume part (i) holds. Then (\ref{mainlem}.1) implies that $\depth_R(N) > \depth (K \ltensor_R N)$. Hence we conclude from  (\ref{mainlem}.2) that $\depth_R(K \ltensor_R N) = \depth_R(M \ltensor_R N)$.

Assume part (ii) holds. Then we have $\depth_R(N) = \dim(R)$. Moreover, as $\q^R(M,N)<\infty$, it follows that $\depth(M \ltensor_R N)\leq \dim(R)$; see  \ref{cxp}(ii).
Thus (\ref{mainlem}.2) yields $\depth(K \ltensor_R N) = \depth (M \ltensor_R N)$.

Finally assume part (iii) holds. Note that $\q^R(M,N)=\q^R(K,N)$; see \ref{lem2.1intro}(ii). Set $q=\q^R(M,N)$. As we assume $\depth_R(\Tor_q^R(M,N)) \le 1$, \ref{rmk:ss} shows that 
\begin{equation}\tag{\ref{mainlem}.3}
\depth_R(M\ltensor_R N)=\depth_R(\Tor_q^R(M,N)) -q.
\end{equation}
On the other hand, we know by \ref{lem2.1intro}(iii) that $\depth_R(\Tor_q^R(K,N)) = \depth_R(\Tor_q^R(M,N))$.
Hence, using \ref{rmk:ss} once more, we see:
\begin{equation}\tag{\ref{mainlem}.4}
\depth_R (K\ltensor_R N)=\depth_R( \Tor_q^R(K,N)) - q.
\end{equation}
Consequently (\ref{mainlem}.3) and (\ref{mainlem}.4) imply that $\depth (K\ltensor_R N)=\depth (M\ltensor_R N)$.
\end{proof}


\begin{prop}\label{mainprop} Let $M$ and $N$ be $R$-modules such that $\q^R(M,N)<\infty$. Assume \emph{at least} one of the following conditions holds:
\begin{enumerate}[\rm(i)]
\item $\depth_R(M) < \depth(R)$.
\item $\depth_R(M)\leq \depth(R)$ and $N$ is maximal Cohen-Macaulay.
\item $\depth_R(M)\leq \depth(R)$ and $\depth_R(\Tor_q(M,N)) \le 1$.
\end{enumerate}
If $\rcid_R(M) < \infty$, then the derived depth formula holds for $M$ and $N$.
\end{prop}

\begin{proof} We set $\rcid_R(M)=c$ and proceed by induction on $c$. If $c=0$, then $\CI_R(M)<\infty$ so that the derived depth formula holds for $M$ and $N$; see \ref{rdim} and \ref{cxp}(vi). 

Next assume $c\geq 1$. Then there is an exact sequence of $R$-modules $0 \to M^{\oplus a} \to K \to \Omega_R^n M^{\oplus b} \to 0$, where $a\geq 1$, $b\geq 1$, and $n\geq 0$ are integers, such that $\rcid_R(K) < \rcid_R(M)$; see \ref{rdim1}(i). Moreover, \ref{lem2.1intro} implies that the conditions (i), (ii), and (iii) hold for $K$. Hence, by the induction hypothesis, we see that the derived depth formula holds for $K$ and $N$. Consequently, the result follows from Lemma \ref{mainlem}.
\end{proof}




Our next result establishes Theorem \ref{thmintro} and generalizes the result of Bergh-Jorgensen advertised as Theorem \ref{thmintroBJ} in the introduction; recall that each module that has reducible complexity also has finite reducing projective dimension, but not necessarily vice versa; see Examples \ref{ahlat1}, \ref{ahlat2}, and \ref{ahlat3}. 

\begin{thm} \label{mainthm} If $R$ is a Cohen-Macaulay local ring and $M$ and $N$ are $R$-modules such that $\q^R(M,N)<\infty$ and $\rcid_R(M)<\infty$, then the derived depth formula holds for $M$ and $N$.
\end{thm}

\begin{proof} Assume $R$ is a Cohen-Macaulay, $\q^R(M,N)<\infty$, and $\rcid_R(M)<\infty$. We may also assume both $M$ and $N$ are nonzero. Then it follows that $\rpd_R(M)<\infty$; see \ref{rdim2}(i). 

If $\depth_R(M) < \depth(R)$, Proposition \ref{mainprop}(i) shows that the derived depth formula holds for $M$ and $N$. Hence we may assume $\depth_R(M)= \depth(R)$, that is, $M$ is maximal Cohen-Macaulay. We set $d=\dim(R)$ and proceed by induction on $d$. 

If $d \le 1$, then Proposition \ref{mainprop}(iii) shows that the derived depth formula holds for $M$ and $N$, and establishes the base case of the induction on $d$. Next we assume $d \ge 2$.

\emph{Case} 1: Assume $\depth_R(N)\geq 1$. Then there is an element $x\in \fm$ which is non zero-divisor on $M$, $N$, and $R$. One can easily show that $\q^{R/xR}(M/xM, N/xN)<\infty$; see, for example, \cite[Page 140, Lemma 2(iii)]{Mat}. Furthermore, we have that $\rpd_{R/xR}(M/xM)\leq \rpd_{R}(M)<\infty$; see \ref{rdim2}(ii). As $R/xR$ is a Cohen-Macaulay local ring with $\dim(R/xR)=d-1<d$, the induction hypothesis on $d$ applied to the pair $(M/xM, N/xN)$ over $R/xR$ shows that the derived depth formula holds for $M/xM$ and $N/xN$ over $R/xR$. Thus the derived depth formula holds for $M$ and $N$ due to \ref{p27}.

\emph{Case} 2: Assume $\depth_R(N)=0$. To deal with this case, we set $\rpd_R(M)=c$ and proceed by induction on $c$. 
If $c=0$, then $\pd_R(M)<\infty$ so that the derived depth formula holds for $M$ and $N$; see \ref{rdim} and \ref{cxp}(vi). So we assume $c\geq 1$. Then it follows from \ref{rdim1}(i) that there is a short exact sequence of $R$-modules $0 \to M^{\oplus a} \to K \to \Omega_R^n M^{\oplus b} \to 0$, where $a\geq 1$, $b\geq 1$, and $n\geq 0$ are integers, such that $\rpd_R(K) < \rpd_R(M)$. Then, in view of \ref{lem2.1intro}, the induction hypothesis implies that the derived depth formula holds for $K$ and $N$. Also we know that $d=\depth_R(M)=\depth_R(K)$; see \ref{lem2.1intro}(i). Hence it suffices to show $\depth_R(M\ltensor_RN)=\depth_R(K\ltensor_RN)$. 

We observe that $\depth_R(K\ltensor_RN)=0$ since the derived depth formula holds for $K$ and $N$, that is, $$d=d+0=\depth_R(K)+\depth_R(N) = \depth(R)+\depth_R(K\ltensor_RN)=d+\depth_R(K\ltensor_RN).$$

As $\depth_R(M)\leq \depth(R)$, and since the derived depth formula holds for $K$ and $N$ over $R$, we have by \ref{rmkmain} that 
$\depth_R(M \ltensor_R N)\geq  \min \{\depth_R(M \ltensor_R N), \depth_R(N)\}=\depth_R(K \ltensor_R N)=0$. Therefore we aim to prove $0=\depth_R(K \ltensor_R N) \geq \depth_R(M \ltensor_R N)$.

Consider a syzygy short exact sequence
\begin{equation} \tag{\ref{mainthm}.1}
0 \to \Omega_RN \to F \to N \to 0,
\end{equation}
where $F$ is a (finitely generated) free $R$-module. As $\depth_R(F)=\depth(R)=d\geq 2>0=\depth_R(N)$, it follows by the depth lemma that $\depth_R(\Omega_RN)=1$. So we can use Case 1 with the pair $(M, \Omega_R N)$, and conclude that $\depth_R(M \ltensor_R \Omega_RN)=\depth_R(M)+\depth_R(\Omega_RN)-d=1$ (recall that
$\depth_R(M)=d)$.

We apply the functor $M \ltensor_R -$ to the exact triangle induced from (\ref{mainthm}.1), and obtain the exact triangle:
\begin{equation} \tag{\ref{mainthm}.2}
M \ltensor_R \Omega_RN \to M \ltensor_R F \to M \ltensor_R N \to (M \ltensor_R \Omega N)[1].
\end{equation}
 
Suppose $\depth_R(M \ltensor_R N)\geq \depth_R(M \ltensor_R \Omega_RN)=1$. Then the derived depth lemma \ref{cxp}(iv)(a) gives the following contradiction $$1=\depth_R(M \ltensor_R \Omega_RN)\geq \min\big\{ \depth_R(M \ltensor_R F), \depth_R(M \ltensor_R N)+1\big\} \geq 2.$$
Consequently, $\depth_R(M \ltensor_R N)<\depth_R(M \ltensor_R \Omega_RN)=1$, that is, $\depth_R(M \ltensor_R N)\leq 0$. This establishes the claim when $\depth_R(N)=0$, and completes the proof of the theorem by induction on $d$.
\end{proof}

We do not know whether the conclusion of Theorem \ref{mainthm} holds over rings that are not necessarily Cohen-Macaulay. Hence the following question is natural:

\begin{ques} Let $R$ be a local ring and let $M$ and $N$ be $R$-modules. If $\rcid_R(M)<\infty$ and $\q^R(M,N)<\infty$, then must the derived depth formula hold for $M$ and $N$?
\end{ques}

\section{Some corollaries and further results}

In this section we obtain some corollaries of Proposition \ref{mainprop} and Theorem \ref{mainthm}, and also consider further consequences of our results from sections 2 and 3. 

%

Auslander \cite{Au} proved the following general version of the depth formula: If $R$ is a local ring, and $M$ and $N$ are $R$-modules such that $\q^R(M,N)<\infty$, then $$\depth_R(M) + \depth_R(N) = \depth(R)+\depth_R\Big(\Tor^R_{\q^R(M,N)}(M,N)\Big)-\q^R(M,N)$$ provided that $\q^R(M,N)=0$, or $\depth_R\Big(\Tor^R_{\q^R(M,N)}(M,N)\Big)\leq 1$. This version of the depth formula was also studied in the literature. One such result is due to Bergh-Jorgensen; see \cite[page 3]{BJD} for the definition of upper reducing degree $\reddeg M$ of $M$.

\begin{chunk} (Bergh-Jorgensen \cite[3.1]{BJD}; see also \cite[3.4(i)]{Be}) \label{thmBJ2} Let $R$ be a Cohen-Macaulay local ring and let $M$ and $N$ be $R$-modules such that $\q^R(M,N)<\infty$. Assume at least one of the following conditions holds:
\begin{enumerate}[\rm(i)] 
\item $\depth_R\Big(\Tor^R_{\q^R(M,N)}(M,N)\Big)=0$.
\item $\q^R(M,N)\geq 1$, $\depth_R\Big(\Tor^R_{\q^R(M,N)}(M,N)\Big)\leq 1$, and $\reddeg M \geq 2$.
\end{enumerate}
If $M$ has reducible complexity, then $$\depth_R(M) + \depth_R(N) = \depth(R)+\depth_R\Big(\Tor^R_{\q^R(M,N)}(M,N)\Big)-\q^R(M,N).$$
\end{chunk}

Next, as a consequence of Proposition \ref{mainprop}(iii) and Theorem \ref{mainthm}, we generalize \ref{thmBJ2}:

\begin{cor}\label{cor5.3} Let $R$ be a Cohen-Macaulay local ring and let $M$ and $N$ be $R$-modules such that $\q^R(M,N)<\infty$ and $\depth_R\Big(\Tor^R_{\q^R(M,N)}(M,N)\Big)\leq 1$. 
If $\rcid_R(M) < \infty$ (for example, $M$ has reducible complexity), then $$\depth_R(M) + \depth_R(N) = \depth(R)+\depth_R\Big(\Tor^R_{\q^R(M,N)}(M,N)\Big)-\q^R(M,N).$$
\end{cor}

\begin{proof} We set $X=M \ltensor_R N$ and $q=\q^R(M,N)$. Then $\depth_R\big(\rH_q(X)\big)=\depth_R\Big(\Tor^R_{q}(M,N)\Big)\leq 1$ and $\sup(X)=q<\infty$. Hence \ref{rmk:ss} implies that $\depth_R(X) = \depth_R\Big(\Tor^R_{q}(M,N)\Big)-q$.  As we know by Proposition \ref{mainprop}(iii) that the derived depth formula holds for $M$ and $N$, the claim follows.
\end{proof}


\begin{rmk} \label{dfinfo} Sadeghi \cite[3.4(ii)]{Sa} proved that the conclusion of Corollary \ref{cor5.3} holds over local rings that are not necessarily Cohen-Macaulay provided that $1\leq \q^R(M,N)\leq \infty$ and $M$ has reducible complexity. Corollary \ref{cor5.3} is also due to Bergh \cite[3.4(i)]{Be} in case $\depth_R\Big(\Tor^R_{\q^R(M,N)}(M,N)\Big)=0$ and $M$ has reducible complexity. \qed
\end{rmk}


We proceed and obtain some consequences of our results from sections 2 and 3. These consequences do not directly imply the depth formula holds, but they contribute to the study of depth of tensor products of modules. 

For their study of the depth formula, Bergh-Jorgensen \cite{BJD} proved the following results.

\begin{chunk} (Bergh-Jorgensen \cite[3.2 and 3.7]{BJD}) \label{thmBJ3} Let $R$ be a Cohen-Macaulay local ring and let $M$ and $N$ be nonzero $R$-modules such that $M$ and $N$ are Tor-independent, that is, $\q^R(M,N)=0$. Assume $M$ has reducible complexity. Assume further $\depth_R(M\otimes_RN)\neq 0$. Then the following hold:
\begin{enumerate}[\rm(i)] 
\item $\depth_R(M)\neq 0$.
\item If $M$ is maximal Cohen-Macaulay and $\Gdim_R(N)<\infty$, then $\depth_R(N)\neq 0$. \qed
\end{enumerate}
\end{chunk}

Thanks to Theorem \ref{mainthm}, we can generalize \ref{thmBJ3} as follows:

\begin{cor}\label{thmBJ3improved} Let $R$ be a Cohen-Macaulay local ring and let $M$ and $N$ be nonzero $R$-modules such that $M$ and $N$ are Tor-independent, that is, $\q^R(M,N)=0$. Assume $\rcid_R(M)<\infty$ (for example, $M$ has reducible complexity). Assume further $\depth_R(M\otimes_RN)\neq 0$. Then $\depth_R(M)\neq 0$ and $\depth_R(N)\neq 0$.
\end{cor}

\begin{proof} We know by Theorem \ref{mainthm} that $\depth_R(M) + \depth_R(N) = \depth(R) + \depth_R(M \ltensor_RN)$. As $\q^R(M,N)=0$, it follows that $\depth_R(M) + \depth_R(N) = \depth(R) + \depth_R(M \otimes_RN)$. Therefore, if $\depth_R(M)=0$, then $\depth_R(N) = \depth(R) + \depth_R(M\otimes_RN)>\depth(R)=\dim(R)$, which is not possible. Hence $\depth_R(M)\neq 0$. Similarly, it follows that $\depth_R(N)\neq 0$.
\end{proof}

The preliminary results recorded in section 2 yield new results on the vanishing of Ext and Tor. As a demonstration, we use \ref{rdim2} and \ref{lem2.1intro}, and prove the following; see \cite[3.3]{Be} for a similar result for modules with reducible complexity.

\begin{prop}\label{rmk5.4} Let $M$ and $N$ be nonzero $R$-modules such that $\q^R(M,N)<\infty$. 
\begin{enumerate}[\rm(i)]
\item If $\rcid_R(M)<\infty$, then $\q^R(M,N) \leq \depth(R)$.
\item If $\rcid_R(M)<\infty$ and $\depth_R(M)\leq \depth(R)$, then $$\depth(R)-\depth_R(M)-\depth_R(N) \leq \q^R(M,N) \leq \depth(R)-\depth_R(M).$$
\end{enumerate}
\end{prop}

\begin{proof} Assume $\rcid_R(M)<\infty$. It follows from \ref{rdim2}(ii) that $\rpd_R(M)< \infty$. Let $\{K_0,\ldots,K_r\}$ be a reducing $\pd$-sequence of $M$; see \ref{rdim}. Then $\pd(K_r)<\infty$ and $\q^R(M,N)=\q^R(K_r,N)$; see \ref{lem2.1intro}(ii). Therefore, part(i) follows since we have
\begin{equation}\tag{\ref{rmk5.4}.1}
\q^R(M,N)=\q^R(K_r,N)\leq\pd_R(K_r)=\depth(R)-\depth_R(K_r)\leq \depth(R).
\end{equation}


Next assume $\depth_R(M)\leq \depth(R)$. Then $\depth_R(M)=\depth(K_r)$ so that (\ref{rmk5.4}.1) yields the inequality $\q^R(M,N) \leq \depth(R)-\depth_R(M)$; see \ref{lem2.1intro}(i). As $\pd_R(K_r)<\infty$, the derived depth formula holds for the pair $(K_r,N)$; see \ref{cxp}(vi). Hence we have:
\begin{equation}\tag{\ref{rmk5.4}.2}
\depth_R(K_r\ltensor_RN) = \depth_R(K_r) + \depth_R(N)- \depth(R).
\end{equation}
Furthermore \ref{cxp}(iii) implies that:
\begin{equation}\tag{\ref{rmk5.4}.3}
\depth_R(K_r\ltensor_RN) \geq-\q^R(M,N).
\end{equation}
Thus (\ref{rmk5.4}.2) and (\ref{rmk5.4}.3) show that $\q^R(M,N) \geq \depth(R) - \depth_R(M) - \depth_R(N)$. Consequently, part (ii) follows.
\end{proof}

We record some corollories of Proposition \ref{rmk5.4}. Part (i) of the next corollary is known by \cite[2.3]{Jor} if $\cid_R(M)<\infty$, while part (ii) follows from \cite[1.1]{Aus} if $\pd_R(M)<\infty$.

\begin{cor} \label{Torvanish} Let $M$ and $N$ be $R$-modules such that $\q^R(M,N)<\infty$ and $\rcid_R(M)<\infty$. 
\begin{enumerate}[\rm(i)]
\item If $\depth(R)=\depth_R(M)$, then $\q^R(M,N)=0$.
\item If $\depth(R)\geq \depth_R(M)$ and $\depth_R(N)=0$, then $\q^R(M,N)=\depth(R)-\depth_R(M)$.
\end{enumerate}
\end{cor}

\begin{proof} The claims follow immediately from Proposition \ref{rmk5.4}(ii).
\end{proof}

\begin{rmk} Let $M$ and $N$ be $R$-modules such that $\q^R(M,N)<\infty$ and $\CI_R(M)<\infty$. It follows that $\q^R(M,N)\leq \CI_R(M)=\depth(R)-\depth_R(M)$; see \cite[2.3]{Jor}. Moreover, if $\depth_R(N)=0$, then $\q^R(M,N)=\depth(R)-\depth_R(M)$ by, for example, Corollary \ref{Torvanish}.
\end{rmk}

Let $M$ and $N$ be $R$-modules such that $N\neq 0$ and $\p^R(M,N)=\sup\{i: \Ext^i_R(M,N)\neq 0\}$. Assume $\rpd_R(M)<\infty$ and $\{K_0,\ldots,K_r\}$ is a reducing $\pd$-sequence of $M$. Then it can be proved along the same lines as in \ref{lem2.1intro}(ii) that $\p^R(M,N)=\p^R(K_r,N)$. As $\pd_R(K_r)<\infty$, it follows from \cite[4.10]{MO} that $\p^R(K_r,N)=\depth(R)-\depth_R(K_r)$. Now, in view of \ref{lem2.1intro}(i), the following result holds:

%
%
\begin{prop} \label{propExt} Let $M$ and $N$ be $R$-modules such that $N\neq 0$, $\depth_R(M)\leq \depth(R)$ and $\p^R(M,N)<\infty$. If $\rcid_R(M)<\infty$, then $\p^R(M,N)=\depth(R)-\depth_R(M)$.
\end{prop}

\begin{rmk} Proposition \ref{propExt} allows us to find examples of modules that do not have finite reducing complete intersection dimension. For example, Jorgensen and Sega \cite[page 475, 476 and 3.3]{NVan} proved that there exists an Artinian Gorenstein local ring $R$ and $R$-modules $M$ and $N$ such that $\Ext^1_R(M,N)\neq 0$ and $\Ext^i_R(M,N)=0$ for all $i\geq 2$. Therefore Proposition \ref{propExt} implies that $\rcid_R(M)=\infty$. \qed
\end{rmk}

%

%
%

Recall that an $R$-module $M$ is torsion, that is, $M$ equals its torsion submodule, if and only if $M_{\fp}=0$ for all $\fp \in \Ass(R)$. The next result is to be compared with \cite[A2]{GORS2}. 

\begin{cor} \label{torsioncor} Let $M$ and $N$ be $R$-modules. Assume $\rcid_R(M)<\infty$. If $\Tor_i^R(M,N)$ is torsion for all $i\gg 0$, then $\Tor_i^R(M,N)$ is torsion for all $i\geq 1$.
\end{cor}

\begin{proof} Assume $\Tor_i^R(M,N)$ is torsion for all $i\gg 0$, and let $\fp \in \Ass(R)$. Then $\q^{R_{\fp}}(M_{\fp}, N_{\fp})<\infty$. Also $\rcid_{R_{\fp}}(M_{\fp})<\infty$ by Lemma \ref{lemrpd}. So Proposition \ref{rmk5.4}(i) shows that $\q^{R_{\fp}}(M_{\fp}, N_{\fp}) \leq \depth(R_{\fp})=0$, that is, $\q^{R_{\fp}}(M_{\fp}, N_{\fp})=0$. This proves that $\Tor_i^R(M,N)$ is torsion for all $i\geq 1$.
\end{proof}

\begin{cor}\label{olsunmu} Let $R$ be a one-dimensional Cohen-Macaulay local ring and let $M$ and $N$ be nonzero $R$-modules. Assume $\q^R(M,N)<\infty$ and $\rcid_R(M)<\infty$. Then the following are equivalent:
\begin{enumerate}[\rm(i)]
\item $\q^R(M,N)=0$.
\item $M$ or $N$ is torsion-free.
\item The depth formula holds, that is, $\depth_R(M)+\depth_R(N)=\depth(R)+\depth_R(M\otimes_RN)$.
\end{enumerate}
\end{cor}

\begin{proof} It follows from Theorem \ref{mainthm} that part (i) implies part (iii). Now assume part (iii) holds. Then we have that $2\geq \depth_R(M)+\depth_R(N)=1+\depth_R(M\otimes_RN)\geq 1$. So $\depth_R(M)=1$ or $\depth_R(N)=1$, that is, $M$ or $N$ is torsion-free. Hence part (iii) implies part (ii). 

Next assume part (ii) holds. We know $\q^R(M,N)\leq 1$; see Proposition \ref{rmk5.4}(i). Suppose $\q^R(M,N)\neq 0$, that is, $\q^R(M,N)=1$. As $\depth_R(\Tor^R_{1}(M,N))\leq 1$, Corollary \ref{cor5.3} yields 
\begin{equation}\tag{\ref{olsunmu}.1}
\depth_R(M) + \depth_R(N) = 1+\depth_R(\Tor_1^R(M,N))-1 = \depth_R(\Tor_1^R(M,N)).
\end{equation}
It follows from Corollary \ref{torsioncor} that $\Tor_1^R(M,N)$ is torsion, that is, $\Tor_1^R(M,N)$ has finite length. Therefore (\ref{olsunmu}.1) gives the equality $\depth_R(M) + \depth_R(N) = 0$. This is not possible since $M$ or $N$ is torsion-free. Consequently $\q^R(M,N)=0$ so that part (ii) implies part (i).
\end{proof}

\begin{rmk} In Corollary \ref{olsunmu}, $M$ and $N$ are not necessarily both torsion-free. If $R$ is a one-dimensional Cohen-Macaulay local ring, $M=R/xR$ for some non zero-divisor $x$ on $R$, and $N$ is a torsion-free $R$-module, then $M$ is torsion, $\q^R(M,N)=0$, and $\depth_R(M)+\depth_R(N)=\depth(R)+\depth_R(M\otimes_RN)$.
\end{rmk}

Next we consider the dependency formula of Jorgensen which is a generalization of the classical Auslander-Buchsbaum formula (recall that $\depth_R(0)=\infty$).

\begin{chunk}(Jorgensen \cite[2.7]{JAB}) \label{JDF} Let $M$ and $N$ be nonzero $R$-modules. If $\q^R(M,N) < \infty$ and $\cid_R(M)<\infty$, then it follows:
\begin{align*}
\q^R(M,N)  & = \sup \Bigl\{  \depth(R_{\fp})-\depth_{R_\fp}(M_\fp) -\depth_{R_\fp}(N_\fp) \mid \fp \in \Supp_R\big(\Tor^R_{\q^R(M,N)}(M,N)\big) \Bigr\}.
\end{align*}
\end{chunk}

Our next result, namely Theorem \ref{thm5.5}, generalizes \ref{JDF} over Cohen-Macaulay rings and establishes Theorem \ref{propnewintro}. 
Note that Theorem \ref{thm5.5} not only improves \ref{JDF} by replacing the assumption $\cid_R(M)<\infty$ with $\rcid_R(M)<\infty$, but also it points out that $\q^R(M,N)$ can be computed by using the (finitely many) prime ideals in $\Ass_R\big(\Tor^R_{\q^R(M,N)}(M,N)\big)$.

\begin{thm}\label{thm5.5} Let $R$ be a Cohen-Macaulay local ring and let $M$ and $N$ be nonzero $R$-modules. If $\q^R(M,N) < \infty$ and $\rcid_R(M)<\infty$, then it follows:
\begin{align*}
\q^R(M,N)  & = \sup \Bigl\{  \depth(R_{\fp})-\depth_{R_\fp}(M_\fp) -\depth_{R_\fp}(N_\fp) \mid \fp \in \Spec(R)  \Bigr\}
 \\ &= \sup\Bigl\{ \depth(R_{\fp})-\depth_{R_\fp}(M_\fp) -\depth_{R_\fp}(N_\fp) \mid \fp \in \Ass_R\big(\Tor^R_{\q^R(M,N)}(M,N)\big)\Bigr\}. 
\end{align*} 
\end{thm}

\begin{proof} We know $\rpd_R(M) < \infty$ since we assume $\rcid_R(M)<\infty$; see \ref{rdim2}(i). Let $\fp\in \Spec(R)$.
Then $\rpd_{R_\fp}(M_\fp)\leq \rpd_R(M)<\infty$ by  \Cref{lemrpd}. Set $X=M_{\fp} \ltensor_{R_\fp} N_{\fp}$. Then it follows that $\sup(X)=\q^{R_\fp}(M_\fp,N_\fp) \leq \q^R(M,N)<\infty$. Thus, by \ref{cxp}(iii), we have that $\depth_R(X)\geq -\q^{R_\fp}(M_\fp,N_\fp)$. This yields:
\begin{equation} \tag{\ref{thm5.5}.1}
\q^R(M,N) \ge \q^{R_\fp}(M_\fp,N_\fp) \ge -\depth (M_\fp \ltensor_{R_\fp} N_\fp).
\end{equation}
As $R_\fp$ is Cohen-Macaulay and $\q^{R_\fp}(M_\fp,N_\fp)<\infty$, Theorem \ref{mainthm} implies that the derived depth formula holds for $(M_\fp,N_\fp)$ over $R_{\fp}$, that is, $\depth_{R_{\fp}}(M_\fp)+\depth_{R_{\fp}}(N_\fp)=\depth(R_\fp)+\depth (M_\fp \ltensor_{R_\fp} N_\fp)$. Hence (\ref{thm5.5}.1) implies that: 
\begin{equation} \tag{\ref{thm5.5}.2}
\q^R(M,N) \ge \depth(R_\fp)-\depth_{R_{\fp}}(M_\fp) -\depth_{R_{\fp}}(N_\fp).
\end{equation}
So $\q^R(M,N) \ge \depth(R_\fq)-\depth_{R_{\fq}}(M_\fq) -\depth_{R_{\fq}}(N_\fq)$ for each $\fq \in \Spec(R)$; see (\ref{thm5.5}.2). Consequently we conclude that
\begin{align}\tag{\ref{thm5.5}.3}
\q^R(M,N) & \ge \sup \Bigl\{ \depth(R_\fp)-\depth_{R_\fp}(M_\fp)-\depth_{R_\fp}(N_\fp) \mid \fp\in\Spec(R)\Bigr\}
 \\ & \geq \notag{}\sup\Bigl\{ \depth(R_{\fp})-\depth_{R_\fp}(M_\fp) -\depth_{R_\fp}(N_\fp) \mid \fp \in \Ass_R\big(\Tor^R_{\q^R(M,N)}(M,N)\big)\Bigr\}.
\end{align}

Next let $\fp\in \Ass\Big(\Tor^R_{\q^R(M,N)}(M,N)\Big)$. Then $\depth_{R_\fp}\Big(\Tor^{R_{\fp}}_{\q^R(M,N)}(M_\fp,N_\fp)\Big)=0$. As $\depth_R(0)=\infty$, it follows that $\Tor^{R_{\fp}}_{\q^R(M,N)}(M_\fp,N_\fp)\neq 0$. Thus we deduce $\q^R(M,N)=\q^{R_\fp}(M_\fp,N_\fp)$. Set $Y=M_{\fp} \ltensor_{R_\fp} N_{\fp}$. Then $\sup(Y)=\q^{R_\fp}(M_\fp,N_\fp)<\infty$ and $\depth_{R_\fp}\Big(\rH_{\q^{R_\fp}(M_\fp,N_\fp)}(Y)\Big)=\depth_{R_\fp}\Big(\Tor^{R_{\fp}}_{\q^R(M,N)}(M_\fp,N_\fp)\Big)=0$. So \ref{rmk:ss} yields:
\begin{equation} \tag{\ref{thm5.5}.4}
\depth_{R_\fp} (M_\fp \ltensor_{R_\fp} N_\fp)=\depth_{R_\fp} (Y)=\depth_{R_\fp}\Big(\rH_{\q^{R_\fp}(M_\fp,N_\fp)}(Y)\Big)-\q^{R_\fp}(M_\fp,N_\fp).
\end{equation}

Recall, by Theorem \ref{mainthm}, we have $\depth_{R_{\fp}}(M_\fp)+\depth_{R_{\fp}}(N_\fp)=\depth(R_\fp)+\depth_{R_{\fp}}(M_\fp \ltensor_{R_\fp} N_\fp)$.
So we conclude from (\ref{thm5.5}.4) that 
\begin{equation} \tag{\ref{thm5.5}.5}
\q^R(M,N)=-\depth_{R_\fp} (M_\fp \ltensor_{R_\fp} N_\fp)=\depth(R_\fp)-\depth_{R_{\fp}}(M_\fp)-\depth_{R_{\fp}}(N_\fp).
\end{equation}
As the equality in (\ref{thm5.5}.5) is true for each associated prime ideal of $\Tor^R_{\q^R(M,N)}(M,N)$, we deduce that
\begin{equation} \tag{\ref{thm5.5}.6}
\q^R(M,N)=\sup\Bigl\{ \depth(R_{\fp})-\depth_{R_\fp}(M_\fp) -\depth_{R_\fp}(N_\fp) \mid \fp \in \Ass_R\big(\Tor^R_{\q^R(M,N)}(M,N)\big)\Bigr\}.
\end{equation}
Now (\ref{thm5.5}.3) and (\ref{thm5.5}.6) give the required conclusion.
\end{proof}


%


\section{Proofs of \ref{lem2.1intro}, \ref{p27}, \ref{rmkmain}, and \ref{rmk:ss}}

This section is devoted to the proofs of the preliminary results stated in section 2. Recall that $R$ denotes a commutative Noetherian local ring with unique maximal ideal $\fm$ and residue field $k$, and all $R$-modules are assumed to be finitely generated.
Moreover, we define an $R$-complex as a complex of finitely generated $R$-modules which is indexed homologically. 

\begin{proof}[Proof of \ref{lem2.1intro}] As part (i) follows by the depth lemma, we proceed with the proofs of parts (ii) and (iii).
We set $q = \q^R(M,N)$. The short exact sequence $0 \to M^{\oplus a} \to K \to \Omega_R^n M^{\oplus b} \to 0$ induces a long exact sequence in Tor modules, which implies the vanishing of $\Tor^R_{i}(K,N)$ for all $i\geq q+1$ and also yields the exact sequence:
\begin{equation}\tag{\ref{lem2.1intro}.1}
0 \to \Tor^R_q(M^{\oplus a},N) \to \Tor^R_q(K,N) \to \Tor^R_q(\Omega_R^nM^{\oplus b},N).
\end{equation}
Hence, as $\Tor^R_{q}(M,N)\neq 0$, it follows that $\Tor^R_q(K,N)\neq0$. This proves $q=\q^R(K,N)$.

Assume $n\geq 1$. Then $\Tor^R_q(\Omega_R^nM^{\oplus b},N)=0$. Hence (\ref{lem2.1intro}.1) yields $\Tor^R_q(M^{\oplus a},N) \cong \Tor^R_q(K,N)$ and shows that $\depth_R\big(\Tor^R_q(M,N)\big)=\depth_R\big(\Tor^R_q(K,N)\big)$, as required.

Next we assume $\depth_R\big( \Tor^R_q(M,N)\big)\leq 1$ and $n=0$. We observe that (\ref{lem2.1intro}.1) yields two exact sequences:
\begin{equation}\tag{\ref{lem2.1intro}.2}
0 \to \Tor^R_q(M^{\oplus a},N) \to \Tor^R_q(K,N) \to X \to 0 \text{ \;and\; }
0 \to X \to \Tor^R_q(M^{\oplus a},N)\label{eq:4.1.4}.
\end{equation}

If $\depth_R( \Tor^R_q(M,N))=0$, then the leftmost exact sequence in (\ref{lem2.1intro}.2) implies that the depth of $\Tor^R_q(K,N)$ is zero. Thus we may assume $\depth_R\big(\Tor^R_q(M,N)\big)=1$. In that case $\depth_R(X)\geq 1$ due to the rightmost exact sequence in (\ref{lem2.1intro}.2). So it follows that $$\depth_R\big(\Tor^R_q(K,N)\big)\geq \min\{ \depth_R\big(\Tor^R_q(M,N)\big), \depth_R(X) \}=1.$$ If $\depth_R\big(\Tor^R_q(K,N)\big)\geq 2=1+\depth_R\big(\Tor^R_q(M,N)\big)$, then the depth lemma applied to the leftmost exact sequence in (\ref{lem2.1intro}.2) implies that $\depth_R\big(\Tor^R_q(M,N)\big)=\depth_R(X)+1\geq 2$, which is not true. Consequently, we conclude that $\depth_R\big(\Tor^R_q(K,N)\big)=1=\depth_R\big(\Tor^R_q(M,N)\big)$.
\end{proof}


We make use of the following lemma for the proof of \ref{p27}.

\begin{lem}\label{lem4} Let $Y$ be a homologically bounded $R$-complex, $x \in \fm$, and let $K=(0\to R \xrightarrow{x}  R \to 0)$ be the Koszul complex on $x$. Then $\depth_R(Y \ltensor_R K) = \depth_R(Y) -1$.
\end{lem}

\begin{proof} Note that we have the exact triangle $Y \xrightarrow{x} Y \to Y \ltensor_R K \to Y[1]$. Given $i\geq 0$, this exact triangle yields the following exact sequence of $R$-modules:
\begin{equation}\tag{\ref{lem4}.1}
 \Ext_R^{i-1}(k, Y)  \to \Ext_R^{i-1}(k, Y \ltensor_R K) \to \Ext_R^i(k, Y) \xrightarrow{x} \Ext_R^i(k, Y).
\end{equation}

Set $n = \depth_R(Y)$. Then, as the map on $\Ext_R^n(k, Y)$ given by multiplication by $x$ is zero, (\ref{lem4}.1) implies that $\Ext_R^{n-1}(k, Y \ltensor_R K) \cong \Ext_R^n(k, Y)$. So $\Ext_R^{n-1}(k, Y \ltensor_R K)\neq 0$. Also, if $i\leq n-2$, then (\ref{lem4}.1) shows that $\Ext_R^i(k, Y \ltensor_R K)=0$. 
Thus $\depth_R(Y \ltensor_R K)=-\sup \rHom_R(k,Y \ltensor_R K)=n-1$.
\end{proof}


\begin{proof}[Proof of  \ref{p27}]  We may assume $M$ and $N$ are nonzero. Hence, since $\q^R(M,N)<\infty$, depths of the complexes in parts (i) and (ii) are finite; see \ref{cxp}(ii) and \ref{lem2.1intro}(ii). Set $S=R/xR$. Then we have the following quasi-isomorphisms: 
\begin{align}
\notag{} (M\ltensor_R N)\ltensor_R S \simeq (M\ltensor_R S)\ltensor_R N & \simeq \left((M\ltensor_R S)\ltensor_S S\right)\ltensor_R N \\ \tag{\ref{p27}.1} & \notag{} \simeq (M\ltensor_R S)\ltensor_S (N\ltensor_R S) \\ & \notag{} \simeq (M/xM)\ltensor_S(N/xN)
\end{align}
Here the last quasi-isomorphism in (\ref{p27}.1) holds since $M\ltensor_R S \simeq M/xM \text{ and }  N\ltensor_R S \simeq N/xN$ due to the fact that $x$ is a non-zero divisor on $R$, $M$, and $N$.

Now (\ref{p27}.1) and Lemma \ref{lem4} yield
\begin{align}\tag{\ref{p27}.2}
\depth_R (M\ltensor_R N)-1= \depth_R\left((M\ltensor_R N)\ltensor_R S\right) & = \depth_R\left((M/xM)\ltensor_S(N/xN)\right) \\ & = \notag{} \depth_S\left((M/xM)\ltensor_S(N/xN)\right).
\end{align}
Here the last equality in (\ref{p27}.2) holds as $\depth_S(X)=\depth_R(X)$ for each $S$-complex $X$; see \cite[5.2(1)]{I}. Moreover, we know $\depth_S(M/xM)=\depth_R (M)-1$ and $\depth_S(N/xN) =\depth_R(N)-1$; see, for example, \cite[1.2.10(d)]{BH}. Consequently we use (\ref{p27}.2) and deduce
\begin{align}
\notag{}  \depth_R(M)+\depth_R(N) = & \depth(R)+\depth_R(M\ltensor_RN)  \Longleftrightarrow \\ 
\notag{}  \depth_R(M)-1 +\depth_R(N)-1 =  &\depth(R)-1+\depth_R(M\ltensor_RN)-1   \Longleftrightarrow \\
\notag{} \depth_{S}(M/xM)+\depth_{S}(N/xN) = & \depth(S)+\depth_{S}(M/xM\ltensor_{S}N/xN). 
\end{align}
This proves the equivalence of part (i) and part (ii).
\end{proof}


To prove \ref{rmkmain}, we first make an observation:

\begin{rmk} \label{212} Let $M$ and $N$ be $R$-modules such that $\q^R(M,N)<\infty$. 

Consider the syzygy short exact sequence $0 \to \Omega_R M \to F \to M \to 0$, where $F$ is a free $R$-module. This sequence yields the exact triangle
\begin{equation}\tag{\ref{212}.1}
\Omega_RM \ltensor_R N \to F\ltensor_R N \to M\ltensor_R N \to (\Omega_RM \ltensor_R N)[1].
\end{equation}
As $\depth_{R}(F\ltensor_R N)=\depth_R(N)$ and $\depth_R\left((M\ltensor_R N)[-1]\right)=\depth_R(M\ltensor_R N)+1$, the derived depth lemma \ref{cxp}(iv)(a) applied to (\ref{212}.1) yields
\[
\depth_R(\Omega_RM \ltensor_R N) \ge \min \big\{\depth_{R}(N), \depth_R(M\ltensor_R N)+1\big\}.
\]
This argument shows that, for each integer $n\geq 1$, we can use induction on $n$ and deduce:
\[
\depth_R(\Omega_R^n M \ltensor_R N) \ge \min\big\{\depth_R(N), \depth_R(M\ltensor_R N)+n\big\}.
\]
\end{rmk}

%
%

Now we can give a proof of \ref{rmkmain}.

\begin{proof}[Proof of \ref{rmkmain}] The given exact sequence $0 \to M^{\oplus a} \to K \to \Omega_R^n M^{\oplus b} \to 0$ yields the exact triangle:
\begin{equation} \tag{\ref{rmkmain}.1}
(M \ltensor_R N)^{\oplus a} \to K \ltensor_R N \to (\Omega_R^n M \ltensor_R N)^{\oplus b} \to (M \ltensor_R N)^{\oplus a}[1].
\end{equation}
Note, as $\q^R(M,N)<\infty$, depths of the complexes in (\ref{rmkmain}.1) are all finite; see \ref{cxp}(ii) and \ref{lem2.1intro}(ii). The derived depth lemma \ref{cxp}(iv)(b) applied to (\ref{rmkmain}.1) establishes the first inequality of the following:
\begin{align*}
\notag{} \depth_R(K \ltensor_R N) &\ge \min \{\depth_R(M \ltensor_R N),  \depth_R(\Omega_R^n M \ltensor_R N)\} \\
\tag{\ref{rmkmain}.2} &\ge \min \{\depth_R(M \ltensor_R N), \depth_R(N),  \depth_R(M \ltensor_R N) + n\} \\
\notag{} &=\min \{\depth_R(M \ltensor_R N), \depth_R(N)\}.
\end{align*}
Here, in (\ref{rmkmain}.2), the second inequality holds due to Remark \ref{212}.

As $\depth_R(M)\leq \depth(R)$, we have $\depth_R(K)=\depth_R(M)$ so that $\depth_R(K) \le \depth(R)$; see \ref{lem2.1intro}(i). Then, since the derived depth formula holds for $K$ and $N$, we obtain
\begin{equation} \tag{\ref{rmkmain}.3}
\depth_R(N) = \depth(R)-\depth_R(K) + \depth_R(K \ltensor_R N)\geq\depth_R(K \ltensor_R N).
\end{equation}

The derived depth lemma \ref{cxp}(iv)(a) applied to (\ref{rmkmain}.1) establishes the first inequality of:
\begin{align}
\notag{} \depth_R(M \ltensor_R N) &\ge \min \{ \depth_R(K \ltensor_R N), \depth_R(\Omega_R^n M \ltensor_R N)+ 1\} \\
\tag{\ref{rmkmain}.4} &\ge \min \{ \depth_R(K \ltensor_R N), \depth_R(M \ltensor_R N) + n + 1, \depth_R(N) + 1\}\\
\notag{} &= \depth_R(K \ltensor_R N).
\end{align}
Here, in (\ref{rmkmain}.4), we obtain the second inequality and the equality by using Remark \ref{212} and (\ref{rmkmain}.3), respectively. Now, since $\depth_R(K \ltensor_R N)$ cannot exceed both $\depth_R(N)$ and $\depth_R(M \ltensor_R N)$, we conclude that $\min\{\depth_R(M \ltensor_R N), \depth_R(N)\} \geq  \depth_R(K \ltensor_R N)$; see (\ref{rmkmain}.3) and  (\ref{rmkmain}.4). Therefore, in view of (\ref{rmkmain}.2), we see that $\depth_R(K \ltensor_R N)=\min\{\depth_R(M \ltensor_R N), \depth_R(N)\}$.
\end{proof}

\begin{proof}[Proof of \ref{rmk:ss}] It suffices to prove that $\depth_R\big(\rH_{s}(X)\big)-s \leq \depth_R\big(\rH_{i}(X)\big)-i$ for each integer $i$ with $i<s$; see \cite[2.3]{I}. If $i$ is such an integer, then $\depth_R\big(\rH_{s}(X)\big)\leq 1 \leq \depth_R\big(\rH_{i}(X)\big)+(s-i)$
since $\depth_R\big(\rH_{i}(X)\big)\geq 0$. Hence the result follows.
\end{proof}

\bibliography{a,b,c}
\bibliographystyle{plain}
\end{document}